\documentclass[12pt,reqno]{amsart}
\usepackage{amsmath, amsthm, amssymb, stmaryrd,mathrsfs}

\advance \topmargin by -\headheight
\advance \topmargin by -\headsep
     
\evensidemargin \oddsidemargin
\newtheorem{theorem}{Theorem}[section]
\newtheorem{lemma}[theorem]{Lemma}

\theoremstyle{definition}

\theoremstyle{remark}

\numberwithin{equation}{section}

\newcommand{\mmod}[1]{\,\,(\text{mod}\,\,#1)}

\def\bfa{{\mathbf a}}
\def\bfb{{\mathbf b}}
\def\bfc{{\mathbf c}}
\def\bfd{{\mathbf d}}

\def\bfx{{\mathbf x}}
\def\bfy{{\mathbf y}}
\def\bfz{{\mathbf z}}

\def\calA{{\mathscr A}}

\def\calN{{\mathscr N}}

\def\dbN{{\mathbb N}}

\def\dbZ{{\mathbb Z}}

\def\grA{{\mathfrak A}}

\def\grm{{\mathfrak m}}\def\grM{{\mathfrak M}}\def\grN{{\mathfrak N}}
\def\grn{{\mathfrak n}}\def\grS{{\mathfrak S}}\def\grP{{\mathfrak P}}

\def\grK{{\mathfrak K}}\def\grp{{\mathfrak p}}

 \def\grX{{\mathfrak X}}

\def\alp{{\alpha}} 
\def\bet{{\beta}}  
 
 \def\Del{{\Delta}}
  
\def\bfeta{{\boldsymbol \eta}} 
  \def\Tet{{\Theta}}

\def\lam{{\lambda}} \def\Lam{{\Lambda}}

\def\d{{\partial}}
\def\eps{\varepsilon}

\def\le{\leqslant} \def\ge{\geqslant}

\def\d{{\,{\rm d}}}

\begin{document}
\title[Diagonal quartic forms]{Pairs of diagonal quartic forms:\\ the non-singular Hasse 
principle}
\author[J\"org Br\"udern]{J\"org Br\"udern}
\address{Mathematisches Institut, Bunsenstrasse 3--5, D-37073 G\"ottingen, Germany}
\email{jbruede@gwdg.de}
\author[Trevor D. Wooley]{Trevor D. Wooley}
\address{Department of Mathematics, Purdue University, 150 N. University Street, West 
Lafayette, IN 47907-2067, USA}
\email{twooley@purdue.edu}
\thanks{The authors are grateful to Akademie der Wissenschaften zu G\"ottingen for 
support during the period when this paper was conceived. The first author was supported 
by Deutsche Forschungsgemeinschaft Project Number 255083470. The second author was 
supported by NSF grants DMS-1854398 and DMS-2001549.}
\date{}

\begin{abstract} We establish the non-singular Hasse Principle for pairs of diagonal quartic 
equations in $22$ or more variables. 
\end{abstract}
\maketitle

\section{Introduction} Given integers $A_j$, $B_j$ with $(A_j,B_j)\neq (0,0)$ 
$(1\le j\le s)$, we consider the pair of Diophantine equations
\begin{equation}\label{1.1}
A_1x_1^4+A_2x_2^4+\ldots +A_sx_s^4=B_1x_1^4+B_2x_2^4+\ldots +B_sx_s^4=0.
\end{equation}
Associated with the coefficients $A_j, B_j$ is the number
$$  {q_0} = {q_0}(\mathbf A, \mathbf B) = \min_{(C,D)\in
\mathbb Z^2\setminus\{(0,0)\}} \text{card}\{1\le j \le s: CA_j+DB_j\neq 0\}. $$
Subject to the conditions $s\ge 22$ and ${q_0}\ge s-7$, our main result in \cite{JEMS1} is 
a quantitative version of the non-singular Hasse principle for the pair of equations 
\eqref{1.1}. This states that whenever the system \eqref{1.1} admits non-singular 
solutions in real numbers and in $p$-adic numbers for each prime $p$, then the number 
$\mathscr N(P)$ of solutions in integers $x_j$, with $|x_j|\le P$ $(1\le j\le s)$, satisfies 
the lower bound $\mathscr N(P)\gg P^{s-8}$.

Although the rank condition on the coefficient matrix expressed through a lower bound for 
$q_0$ may appear unnatural and restrictive, it cannot be abandoned entirely. We 
demonstrate in \S7 that whenever $s\ge 9$ the  pair of equations
\begin{equation}\label{1.2}
x_1^4+x_2^4-6x_3^4-12 x_4^4=x_4^4 - 7x_5^4-5x_6^4-3x_7^4-
\sum_{j=8}^s x_j^4=0 \end{equation}
has non-singular solutions in all completions of the rationals, but only the zero solution in 
integers. Thus, the non-singular Hasse principle fails for this pair with $q_0=4$. As 
promised in \cite{JEMS2}, we return to the subject here and relax the condition on $q_0$ 
to one that is independent of $s$.

\begin{theorem}\label{theorem1.1}
Let $s\ge 22$ and ${q_0} \ge 12$. Then provided that the system \eqref{1.1} has 
non-singular solutions in each completion of the rational numbers, one has 
$\calN(P)\gg P^{s-8}$.
\end{theorem}

For an assessment of the strength of this result, we refer to important work of Vaughan 
\cite{V89}, implying that a single equation
$$A_1 x_1^4+ A_2x_2^4 + \ldots +A_{t}x_{t}^4 =0$$
obeys the Hasse principle when $t\ge 12$. For such equations with $t= 11$ this is not yet 
known. Thus, as detailed in the introduction of the precedessor to this memoir 
\cite{JEMS1}, our result breaks with the established rule that systems of Diophantine 
equations deny treatment unless the number of variables is at least as large as the sum 
of the number of variables needed to solve each of the equations individually. Also, if the 
conclusion of Theorem \ref{theorem1.1} were to be true in all cases where $q_0=11$, 
then the Hasse principle would follow for diagonal quartic equations in $11$ variables. To 
see this, consider the pair of equations
\begin{equation}\label{1.3}
A_1 x_1^4+ A_2x_2^4 + \ldots +A_{11}x_{11}^4 =
x_{11}^4-x_{12}^4 + x_{13}^4-x_{14}^4 + x_{15}^4 + \ldots + x_{22}^4 =0,
\end{equation}
in which $A_j\neq 0$ for $1\le j\le 11$. Then $q_0=11$. If $(x_1,\ldots ,x_{11})$ is a 
non-trivial solution to the equation
\begin{equation}\label{1.4}
A_1x_1^4+A_2x_2^4+\ldots +A_{11}x_{11}^4=0
\end{equation} 
in a  completion of the rationals, then in the above system we choose $x_{12}=x_{11}$, 
$x_{13}=x_{14}=1$ and $x_j=0$ for $15\le j\le 22$ to produce a non-singular solution 
$\bfx$ to the system (\ref{1.3}). The putative extension of Theorem \ref{theorem1.1} to 
pairs of equations with $q_0=11$ implies that there are $\gg P^{14}$ solutions in 
integers $x_j$ with $|x_j|\le P$ $(1\le j\le 22)$. This is only possible if some of these 
solutions have some variable $x_j$ with $1\le j\le 11$ non-zero. Consequently, the 
quantitative non-singular Hasse principle for the above system implies the Hasse principle 
for the equation (\ref{1.4}), as claimed. In light of this observation, it appears difficult to 
further relax the condition on $q_0$ in Theorem \ref{theorem1.1}.\par

Three novel ingredients are required to establish Theorem \ref{theorem1.1}, of which two 
were not available at the time when \cite{JEMS1} was written. We proceed to describe 
discriminating invariants associated with the system \eqref{1.1} that we shall use to 
identify various cases that require treatment by three different methods. With each pair of 
coefficients $(A_j,B_j)$ occurring in the system \eqref{1.1}, we associate the linear form 
$\Lam_j=\Lam_j(\alp,\bet)$ defined by
\begin{equation}\label{1.5}
\Lam_j=A_j\alp+B_j\bet\quad (1\le j\le s).
\end{equation}
Recall in this context that $(A_j,B_j)\ne (0,0)$ $(1\le j\le s)$. We refer to indices $i$ and 
$j$ as being {\em equivalent} when there exists a non-zero rational number $\lam$ with 
$\Lam_i=\lam\Lam_j$. Suppose that the equivalence relation thus defined amongst the 
indices $1,2,\ldots,s$ has exactly $t$ equivalence classes, and that the number of indices 
in these classes are $r_1\ge r_2\ge \ldots \ge r_t$. If we assume that the system 
\eqref{1.1} admits a non-singular real solution, then we must have $t\ge 2$. In the 
interest of notational simplicity, we put $n=r_1$, $m=r_2$ and suppose that $x_j$ 
$(1\le j\le n)$ are the variables counted by $r_1$, and that $y_j=x_{n+j}$ $(1\le j\le m)$ 
are the variables counted by $r_2$. The remaining variables (if any) we denote by 
$z_j=x_{n+m+j}$ $(1\le j\le l)$. Then, by taking suitable linear combinations of the two 
equations in \eqref{1.1}, we pass to an equivalent system of the shape
\begin{equation}
\label{1.6} \left. \begin{array}{rrccc}
a_1x_1^4+\ldots +a_n x_n^4+&\!\!\!  &\!\!\!c_1z_1^4+\ldots +c_lz_l^4=&\!\!\!0 \\
&\!\!\! b_1y_1^4+\ldots +b_my_m^4+&\!\!\!d_1z_1^4+\ldots +d_lz_l^4 =&\!\!\!0
\end{array}\right\} ,
\end{equation}
for suitable non-zero integers $a_i,b_j,c_k,d_k$. Note that $q_0$ is invariant with respect 
to this operation, so that $q_0=s-n$. When $m$ is not too small, it is possible to adopt 
just the simplest ideas from our work on similarly shaped systems of cubic forms 
\cite{BWCrelle} to derive a considerable strengthening of Theorem \ref{theorem1.1}.

\begin{theorem}\label{theorem1.2}
Suppose that $n\ge m\ge 6$ and $m+l\ge 12$, and that the pair of equations \eqref{1.6} 
has non-singular solutions in each completion of the rational numbers. Then 
$\calN(P)\gg P^{s-8}$. 
\end{theorem}

Note that Theorem \ref{theorem1.2} applies to pairs of quartic forms with $n=m=l=6$ 
and $s=18$, for example.  Evidently, Theorem \ref{theorem1.2} implies Theorem 
\ref{theorem1.1} in all cases where $m\ge 6$. Our approach in the remaining cases is 
radically different and depends on a new entangled two-dimensional moment estimate 
that involves $18$ smooth Weyl sums (see Theorem \ref{theorem2.4} below). The 
arithmetic harmonic analysis provides an estimate in terms of eighth and tenth moments 
of smooth Weyl sums, and only the most recent bound for the latter 
\cite[Theorem 1.3]{JEMS2} is of strength sufficient for the application in this paper. The 
new $18$th moment estimate replaces a similar $21$st moment estimate in our previous 
work \cite[Theorem 2.3]{JEMS1}, and demonstrates more flexibility in absorbing larger 
values of the parameters $r_j$ introduced earlier. This suffices to deal with all cases of 
Theorem \ref{theorem1.1} where  $q_0\ge 13$. If one directs this second line of attack 
to the cases with $q_0=12$ then difficulties arise on which we comment later, in the 
course of the argument. This prompted the authors \cite{BWnew} to employ the second 
author's breaking convexity devices \cite{WInv} to the installment of a minor arc moment 
estimate for biquadratic smooth Weyl sums. This, our third new tool, seems indispensable 
for a successful treatment of those cases of Theorem \ref{theorem1.1} where $q_0=12$ 
and $m$ is small, but also helps us along to give a slick proof of Theorem 
\ref{theorem1.2}.\par

{\em Notation}. Our basic parameter is $P$, a sufficiently large real number. Implicit 
constants in Vinogradov's familiar symbols $\ll$ and $\gg$ may depend on $s$ and 
$\varepsilon$ as well as ambient coefficients such as those in the system \eqref{1.1}. 
In this paper, whenever $\varepsilon$ appears in a statement we assert that the 
statement holds for each positive real value assigned to $\varepsilon$.    

\section{Mean value estimates: old and new} 
In this section we present the new entangled moment estimate, preceded by  a summary 
of more familiar one dimensional mean values. We begin with some preparatory notation.

\par When $P$ and $R$ are real numbers with $1\le R\le P$, we define the set of smooth 
numbers $\calA(P,R)$ by
$$\calA(P,R)=\{ u\in \dbZ\cap[1,P]:\text{$p$ prime and $p|u\Rightarrow p\le R$}\}.$$
We then define the Weyl sum $h(\alp)=h(\alp;P,R)$ by
$$h(\alp;P,R)=\sum_{x\in \calA(P,R)}e(\alp x^4).$$
It is convenient to refer to an exponent $\Del_t$ as {\it admissible} if there exists a 
positive number $\eta$ such that, whenever $1\le R\le P^\eta$, one has
\begin{equation}\label{2.1}
\int_0^1|h(\alp;P,R)|^t\d\alp\ll P^{t-4+\Del_t}.
\end{equation}

\begin{lemma}\label{lemma2.1} The exponents $\Del_6=1.1835$, $\Del_{8}=0.5942$,
 $\Del_{10}= 0.1992$ and $\Delta_t=0$, for $t\ge 11.9560$, are admissible.
\end{lemma}

\begin{proof} For $t=6$ and $8$, this follows from \cite[Theorem 2]{BW2000}. For 
$t=10$ see \cite[Theorem 1.3]{JEMS2}. The last statement is a consequence of 
\cite[Theorem 1.2]{BWnew}.
\end{proof} 

Throughout the rest of this paper, we fix $\eta>0$ to be in accordance with Lemma 
\ref{lemma2.1} and the estimate (\ref{2.1}).\par

In the following chain of auxiliary estimates let $\mathscr Z$ denote a set of $Z$ integers 
and put
\begin{equation}\label{2.2}
K(\alpha)=\sum_{z\in\mathscr Z}e(\alpha z).
\end{equation}
For $\nu=1$ or $2$ put
$$ J_\nu= \int_0^1 |h(\alpha)^{2\nu}K(\alpha)^2| \,\mathrm d\alpha . $$

\begin{lemma}\label{lemma2.2} One has
$$J_1\ll PZ+P^{1/2+\varepsilon}Z^{3/2}\quad \text{and}\quad J_2\ll 
P^3Z+P^{2+\varepsilon}Z^{3/2}.$$
\end{lemma} 
\begin{proof} Both estimates are instances of \cite[Lemma 6.1]{KW}. Here we note that, 
although in the proof of the latter it is supposed that $\mathscr Z\subseteq [1,P^4]$, it is 
readily seen that this constraint is redundant in the argument.
\end{proof}

We next introduce a simple form of the entangled moment. For fixed integers $a,b,c,d$, 
define
$$I(a,b,c,d)=\int_0^1 \!\! \int_0^1 |h(a\alpha)h(b\beta)h(c\alpha+d\beta)|^6
\,\mathrm d\alpha\,\mathrm d\beta .$$
For the discussion to come, we fix choices for admissible exponents according to Lemma 
\ref{lemma2.1}, and we fix a choice for the parameters $\tau$ and $\tau_1$ satisfying
$$0<\tau<\tau_1<(1-\Del_8-2\Del_{10})/4.$$
Thus, we may suppose that $0<\tau<0.00185$. We stress that the positivity of $\tau$ is 
assured only because $\Del_8+2\Del_{10}<1$, an inequality obtained by the narrowest of 
margins by virtue of our recent work \cite[Theorem 1.3]{JEMS2}.

\begin{lemma}\label{lemma2.3} Let $a,b,c,d$ be non-zero integers. Then 
$I(a,b,c,d)\ll P^{21/2 -2\tau}$.
\end{lemma}
\begin{proof}
Let $\psi(n)$ denote the number of integers $x_j\in\mathscr A(P,P^\eta)$ $(1\le j\le 6)$ 
with
$$ x_1^4+x_2^4+x_3^4-x_4^4-x_5^4-x_6^4 = n.$$
Then
\begin{equation}\label{2.3}
|h(\gamma)|^6 = \sum_{n\in\mathbb Z} \psi(n) e(\gamma n).
\end{equation}
Note that $\psi(n)=0$ holds for all $|n|>3P^4$, so that the sum in \eqref{2.3} is in fact 
restricted to a finite range. Also, one has $\psi(n)=\psi(-n)$ for all $n$.\par

Write $I=I(a,b,c,d)$. Then, by orthogonality and \eqref{2.3}, we see that
\begin{align}
I=&\sum_{\substack{an_1=cn_3 \\ bn_2=dn_3}} \psi(n_1)\psi(n_2)\psi(n_3) \nonumber\\
\le &\sum_{\substack{an_1=cn_3 \\ bn_2=dn_3}} \big(\psi(n_1)^3+ 
\psi(n_2)^3+\psi(n_3)^3\big) \le 3 \sum_{n\in\mathbb Z} \psi(n)^3 .\label{2.4}
\end{align}
Again by orthogonality, we have
\begin{equation}\label{2.5}
\psi(n)=\int_0^1 |h(\alpha)|^6e(-\alpha n)\,d\alpha.
\end{equation}
Hence, by Lemma \ref{lemma2.1}, for each $n\in \dbZ$ we have
\begin{equation}\label{2.6}
\psi(n)\le \psi(0) \ll P^{2+\Delta_6}.
\end{equation}
Since $3(2+\Del_6)<10$, we may combine the terms with $n$ and $-n$ in \eqref{2.4} to 
deduce that 
$$I\ll P^{10}+\sum_{n=1}^\infty \psi(n)^3.$$

\par Let $T\ge 1$, and define
$${\mathscr Z}_T=\{n\in\mathbb N : T\le \psi(n) <2T\}.$$
When $P$ is large, it follows from \eqref{2.6} that this set is empty unless 
$T\le P^{3.185}$. Since all natural numbers $n$ with $\psi(n)\neq 0$ belong to some one 
of these sets ${\mathscr Z}_T$ as $T$ runs through powers of 2, it follows that there is a 
choice for $T$ with $1\le T \le P^{3.185}$ for which 
\begin{equation}\label{2.7}
I\ll P^{10}+P^\varepsilon T^3 \text{card}({\mathscr Z}_T). 
\end{equation}
Three different arguments are now required, depending on the size of $T$. We take 
${\mathscr Z}={\mathscr Z}_T$ in \eqref{2.2}, write $Z=\text{card}(\mathscr Z)$ as 
before, and then deduce from \eqref{2.2} and \eqref{2.5} that
\begin{equation}\label{2.8}
TZ\le \int_0^1 |h(\alpha)|^6 K(-\alpha)\,\mathrm d\alpha.
\end{equation}

Our first approach to bounding $Z$ applies H\"older's inequality on the right hand side of 
\eqref{2.8} to obtain 
\begin{align*}
TZ\le  J_2^{1/3}\biggl(\int_0^1|K(\alp)|^2\d\alp\biggr)^{1/6}
\biggl(\int_0^1|h(\alp)|^{8}\d\alp\biggr)^{1/6}
\biggl( \int_0^1|h(\alp)|^{10}\d\alp \biggr)^{1/3}.
\end{align*}
Recalling Lemma \ref{lemma2.2}, Parseval's identity and \eqref{2.1}, we deduce  that
$$TZ\ll (P^3Z+P^{2+\eps}Z^{3/2})^{1/3}(Z)^{1/6}(P^{4+\Del_{8}})^{1/6}
(P^{6+\Delta_{10}})^{1/3}.$$
Since our hypothesis on $\tau_1$ ensures that $\Delta_8+2\Delta_{10}<1-4\tau_1$, we 
infer that 
$$T^3Z\ll TP^{(23-4\tau_1)/3}+P^{21/2-2\tau_1}.$$
In view of (\ref{2.7}), this bound provides an acceptable estimate for $I$ should $T$ be in 
the range $1\le T\le P^{(17-4\tau_1)/6}$.\par

Alternatively, we may apply Schwarz's inequality to the right hand side of \eqref{2.8} to 
infer that
$$TZ\le J_2^{1/2}\biggl( \int_0^1|h(\alp)|^{8}\d\alp \biggr)^{1/2}.$$
By Lemma \ref{lemma2.2} and \eqref{2.1}, this gives
$$TZ\le (P^3Z+P^{2+\eps}Z^{3/2})^{1/2} (P^{4+\Delta_8})^{1/2},$$
an estimate that disentangles to yield the bound
$$T^3Z\ll TP^{7+\Delta_8}+T^{-1}P^{12+2\Delta_8+\eps}.$$
On recalling that Lemma \ref{lemma2.1} permits us the assumption that $\Del_8\le 3/5$, 
we deduce by reference to (\ref{2.7}) that our alternative bound for $T^3Z$ produces an 
acceptable estimate for $I$ whenever 
$P^{27/10+2\tau_1}\le T\le P^{7/2-\Del_8-2\tau_1}$.   

The first two approaches just described handle overlapping ranges for $T$, leaving only the 
range $P^{7/2-\Del_8-2\tau_1}<T\le P^{3.185}$ to be addressed. Our final method 
eliminates these large values of $T$ from consideration. Again, we apply Schwarz's 
inequality to the right hand side of \eqref{2.8} to obtain
$$TZ\le J_1^{1/2}\biggl( \int_0^1|h(\alp)|^{10}\d\alp \biggr)^{1/2}.$$
By Lemma \ref{lemma2.2} and \eqref{2.1}, we infer that
$$TZ\ll (PZ+P^{1/2+\eps}Z^{3/2})^{1/2}(P^{6+\Delta_{10}})^{1/2}.$$
On recalling our hypothesis that $\Del_8+2\Del_{10}<1-4\tau_1$, this leads to the bound
$$T^3Z\ll TP^{7+\Delta_{10}}+T^{-1}P^{14-\Del_8-4\tau_1}.$$
By Lemma \ref{lemma2.1}, we may suppose that $\Del_{10}+3.185<3.4$, and thus in 
the final range $P^{7/2-\Del_8-2\tau_1}<T\le P^{3.185}$ we again obtain an acceptable 
upper bound for $I$. This completes the proof of the lemma. 
\end{proof}

\begin{theorem}\label{theorem2.4}
Suppose that $C_i$ and $D_i$ $(1\le i\le 3)$ are integers having the property that any two 
of the three linear forms $\mathrm M_i=C_i\alpha+D_i\beta$ $(1\le i\le 3)$ are linearly 
independent. Then
$$\int_0^1\!\!\int_0^1 |h(\mathrm M_1)h(\mathrm M_2)h(\mathrm M_3)|^6\,\mathrm 
d\alpha\,\mathrm d\beta \ll P^{21/2 -2\tau}.$$
\end{theorem}

\begin{proof} By \eqref{2.3} and orthogonality, the integral in question is equal to
$$\sum \psi(n_1)\psi(n_2)\psi(n_3),$$
with the sum extended over all $n_i\in \mathbb Z$ $(1\le i\le 3)$ satisfying
$$ C_1n_1+C_2n_2+C_3n_3 =  D_1n_1+D_2n_2+D_3n_3 =0.$$
All $2\times 2$ minors of the coefficient matrix associated with this pair of linear equations 
are non-singular, so there is an equivalent system $an_1=cn_3$, $bn_2=dn_3$ with
non-zero integers $a,b,c,d$. The conclusion of the theorem now follows from Lemma 
\ref{lemma2.3} via \eqref{2.4}.\end{proof}    

\section{The circle method}
In this section we prepare the ground for a circle method approach to Theorems 
\ref{theorem1.1} and \ref{theorem1.2}. We shall assume throughout, as we may, that the 
system \eqref{1.1} is already in the form \eqref{1.6}. The linear forms 
$\Lam_j=\Lam_j(\alp,\bet)$ defined by (\ref{1.5}) can then be written also as
\begin{align}\label{3.2}
\begin{aligned}\Lambda_j = a_j\alpha\;\;&(1\le j\le n), \quad \Lambda_{n+j}=b_j\beta\;\;
(1\le j\le m), \\ 
&\Lambda_{n+m+j}=c_j\alpha+d_j\beta\;\;(1\le j\le l).
\end{aligned}
\end{align}
We fix real numbers $\eta_j\in (0,1]$ for $1\le j\le s$, form the generating function
\begin{equation}\label{3.3}
\mathscr F_\bfeta (\alpha,\beta) = \prod_{j=1}^sh(\Lambda_j;P,P^{\eta_j})
\end{equation}
and define
\begin{equation}\label{3.4}
\mathscr N_\bfeta  (P) = \int_0^1\!\!\int_0^1 \mathscr F_\bfeta (\alpha,\beta)
\,\mathrm d\alpha\,\mathrm d\beta.
\end{equation}
By orthogonality, it follows from the definition of $\mathscr N(P)$ that
\begin{equation}\label{3.5}
\mathscr N(P) \ge \mathscr N_\bfeta  (P).
\end{equation}
The specific choice of $\bfeta$ will depend on the coefficients in \eqref{1.1} and 
\eqref{1.6}. Here we analyse the contribution of the major arcs to the integral \eqref{3.4} 
for a generic choice of $\bfeta$. This is largely standard but the relatively low number of 
variables available to us in Theorem \ref{theorem1.2} calls for a brief account.\par 

We begin by defining the singular integral associated with the system of equations 
\eqref{1.1}. This features the integral
\begin{equation}\label{3.6}
v(\gamma) = \int_0^P e(\gamma\xi^4)\,\mathrm d\xi
\end{equation}
from which we build the generating function 
\begin{equation}\label{3.7}
V(\alpha,\beta) = \prod_{j=1}^s v(\Lambda_j),
\end{equation}
and, for $X\ge 1$, the truncated singular integral
$$\mathfrak I(X)=\int_{-XP^{-4}}^{XP^{-4}}\int_{-XP^{-4}}^{XP^{-4}}V(\alpha,\beta) 
\,\mathrm d\alpha\,\mathrm d\beta.$$

\begin{lemma}\label{lemma3.1}
Suppose that $q_0(\mathbf A,\mathbf B)\ge 9$. Then $V$ is integrable over 
$\mathbb R^2$ and
\begin{equation}\label{3.8}
\int_{-\infty}^\infty\int_{-\infty}^\infty|V(\alpha,\beta)|\,\mathrm d\alpha\,\mathrm d\beta 
\ll P^{s-8}. 
\end{equation}
Moreover, the limit
$$\mathfrak I=\lim_{X\to\infty}\mathfrak I(X)$$
exists, and whenever $X\ge 1$ one has
\begin{equation}\label{3.9}
\mathfrak I-\mathfrak I(X)\ll P^{s-8}X^{-1/4}.
\end{equation}
Finally, if the system \eqref{1.1} has a non-singular real solution, then 
$\mathfrak I\gg P^{s-8}$.
\end{lemma}

\begin{proof} There are at least five disjoint subsets $\{i,j\}\subset\{1,2,\ldots,s\}$ having 
the property that the two linear forms $\Lambda_i$, $ \Lambda_j$ are linearly independent. 
To see this, let us suppose that at most four such pairs can be formed. Then at least $s-8$ 
of the forms $\Lambda_1,\ldots ,\Lambda_s$ lie in a one-dimensional space, showing that 
$n\ge s-8$. But $s=n+q_0$, so that $q_0\le 8$, which is not the case.\par

Based on this observation we temporarily (only within this proof) relabel the indices of the 
forms $\Lambda_1,\ldots ,\Lambda_s$ so as to arrange that for $1\le j\le 5$, the pairs 
$\Lambda_{2j-1}$, $\Lambda_{2j}$ are linearly independent. We use the familiar inequality
\begin{equation}\label{3.10}
|z_1z_2\cdots z_r|\ll |z_1|^r+\ldots +|z_r|^r
\end{equation}
with $r=5$ and find from \eqref{3.7} that
\begin{equation}\label{3.11}
V(\alpha,\beta)\ll P^{s-10}\sum_{j=1}^5|v(\Lambda_{2j-1})v(\Lambda_{2j})|^5. 
\end{equation}
By \eqref{3.6} and integration by parts, we have
\begin{equation}\label{3.12}
v(\gamma)\ll P(1+P^4|\gamma|)^{-1/4} . 
\end{equation}
This shows that $v^5$ is integrable with
\begin{equation}\label{3.13}
\int_{-\infty}^\infty |v(\gamma)|^5\,\mathrm d\gamma \ll P.
\end{equation}
The linear change of variable from $(\alpha,\beta)$ to $(\Lambda_{2j-1}, \Lambda_{2j})$ 
shows $|v(\Lambda_{2j-1})v(\Lambda_{2j})|^5$ to be integrable over $\mathbb R^2$.  
Since $V$ is continuous, its integrability follows from dominated convergence and 
\eqref{3.11}, while \eqref{3.13} in combination with (\ref{3.11}) implies \eqref{3.8}. The 
existence of the limit $\mathfrak I$ is now immediate.\par

Next, with
$$\mathscr B(X)=\{(\alpha,\beta)\in\mathbb R^2: \max \{|\alpha|,|\beta|\}\ge XP^{-4}\},
$$
we have
$$\mathfrak I-\mathfrak I(X)\ll \int_{\mathscr B(X)}|V(\alpha,\beta)|\,\mathrm d\alpha
\,\mathrm d\beta.$$
We apply \eqref{3.11} and observe that whenever $(\alpha,\beta)\in\mathscr B(X)$, one 
has
$$\max \{|\Lambda_{2j-1}|,|\Lambda_{2j}|\} \gg XP^{-4}.$$
Hence, by linear changes of variables as before, and with a suitable $C>0$ depending only 
on $\mathbf A$ and $\mathbf B$, we find that
$$\mathfrak I-\mathfrak I(X)\ll P^{s-10}\int_{-\infty}^\infty \int_{CXP^{-4}}^\infty 
|v(\gamma)v(\delta)|^5\,\mathrm d\gamma\,\mathrm d\delta,$$
and \eqref{3.9} follows from \eqref{3.12}.\par

Finally, in order to derive the lower bound for $\mathfrak I$, one first observes that if 
\eqref{1.1} has a non-singular real solution, then there is such a solution with all of its 
coordinates in the interval $(0,1)$ (see the discussion on page 2894 of \cite{JEMS1}). From 
here one may follow the argument used to prove \cite[Lemma 13]{BW07}, {\it mutatis 
mutandis}, so as to establish the final conclusion of the lemma.
\end{proof}

We now turn to the singular series. Its germ is the Gauss sum
$$S(q,a)=\sum_{r=1}^qe(ar^4/q)$$
allowing us to define the generating functions
$$T(q,a,b)=\prod_{j=1}^s S(q,\Lambda_j(a,b)),$$
$$U(q) = q^{-s}\underset{(a,b,q)=1}{\sum_{a=1}^q\sum_{b=1}^q}\,T(q,a,b)\quad 
\text{and}\quad U^\dagger (q) = q^{-s}
\underset{(a,b,q)=1}{\sum_{a=1}^q\sum_{b=1}^q}\, |T(q,a,b)|.$$

\begin{lemma}\label{lemma3.2}
Suppose that $s\ge 16$ and ${q_0}(\mathbf A,\mathbf B)\ge 12$. Then 
$U^\dagger (q)\ll q^{\varepsilon-2}$.
\end{lemma}

\begin{proof} From \cite[Theorem 4.2]{hlm} we deduce that
$$q^{-1}S(q,u)\ll q^{-1/4}(q,u)^{1/4},$$
so that
$$U^\dagger (q)\ll q^{-s/4}\underset{(a,b,q)=1}{\sum_{a=1}^q\sum_{b=1}^q}
\prod_{j=1}^s \big(q, \Lambda_j(a,b)\big)^{1/4}.$$
Recall the data $r_1,\ldots,r_t$ introduced in the preamble to Theorem \ref{theorem1.2}. 
Now following through the argument on page 890 of \cite{BW07} in combination with 
\cite[Lemma 11]{BW07} one finds that there is a number $\Delta $ depending only on the 
coefficients $\mathbf A$ and $\mathbf B$ such that 
$$U^\dagger(q)\ll q^{2-s/4}\sum_{\substack{v_1,\ldots ,v_t\\ v_1v_2\cdots v_t|\Del q}}
v_1^{(r_1-4)/4}\cdots v_t^{(r_t-4)/4}.$$
But $r_j\le r_1=n \le s-12$ $(1\le j\le t)$ and $s\ge 16$, and thus
$$U^\dagger (q)\ll q^{2-s/4}\sum_{\substack{v_1,\ldots ,v_t\\ v_1v_2\cdots v_t|\Del q}}
(v_1\cdots v_t)^{(s-16)/4}\ll q^{\eps-2}.$$
This completes the proof of the lemma. 
\end{proof}

The truncated singular series is defined by
$$\grS(X)=\sum_{1\le q\le X} U(q).$$
The preceding lemma has the following corollary.

\begin{lemma}\label{lemma3.3} Suppose that $s\ge 16$ and 
$q_0(\mathbf A,\mathbf B)\ge 12$. Then the series
$$\grS=\sum_{q=1}^\infty U(q)$$
converges absolutely, and $\grS(X)-\grS\ll X^{\varepsilon-1}$. Moreover, if, for each prime 
$p$, the system of equations \eqref{1.1} has a non-singular solution in $\mathbb Q_p$, 
then $\grS\gg 1$. 
\end{lemma}

\begin{proof} Only the last statement requires justification, for which we may follow 
through the argument of the proof of \cite[Lemma 12]{BW07} with obvious adjustments.
\end{proof}

Now is the time to define the major arcs in our Hardy-Littlewood dissection. We suppose 
now that $0<\tau<10^{-4}$ and put $Q= (\log P)^\tau$. We require dissections of 
dimension one and two. First, when $a\in \dbZ$ and $q\in \dbN$ we define
$$\grM(q,a)=\{\alp\in [0,1] :|\alp-a/q|\le QP^{-4}\}.$$
We then take $\grM$ to be the union of the intervals $\grM(q,a)$ with $0\le a\le q\le Q$ 
and $(q,a)=1$, and we put $\grm=[0,1]\setminus \grM$. Second, when $a,b\in \dbZ$ and 
$q\in \dbN$, we define
$$\grN(q,a,b)=\{(\alp,\bet)\in [0,1)^2:\text{$|\alp-a/q|\le QP^{-4}$ and 
$|\bet-b/q|\le QP^{-4}$}\}.$$
We take $\grN$ to be the union of the rectangles $\grN(q,a,b)$ with $0\le a,b\le q\le Q$ 
and $(q,a,b)=1$, and we put $\grn=[0,1)^2\setminus \grN$. 

\begin{lemma}\label{lemma3.4}
Suppose that $s\ge 16$ and $q_0(\mathbf A,\mathbf B)\ge 12$. Then there exists a number 
$\rho>0$ such that both
$$\iint_\grN \mathscr F_\bfeta(\alpha,\beta) \,\mathrm d\alpha\,\mathrm d\beta 
=\rho\grS\mathfrak I+O(P^{s-8}(\log P)^{-\tau/4})$$
and
$$\int_\grM \int_\grM \mathscr F_\bfeta(\alpha,\beta) \,\mathrm d\alpha\,\mathrm d\beta 
=\rho\grS\mathfrak I+O(P^{s-8}(\log P)^{-\tau/4}).$$
Moreover, if $\eta_j=1$ for $1\le j\le s$, then $\rho=1$.
\end{lemma}

\begin{proof} We begin by establishing the first asymptotic relation asserted in the 
statement. Let $(\alpha,\beta)\in \grN(q,a,b)$ and $q\le (\log P)^{1/10}$. Repeated use of 
\cite[Lemma 8.5]{Wsim2} shows that there is a number $\rho>0$ with
\begin{equation}\label{3.14}
\mathscr F_\bfeta(\alpha,\beta)=\rho q^{-s}T(q,a,b)V(\alpha-a/q,\beta-b/q)+
O(P^s(\log P)^{-1/2}), 
\end{equation}
while \cite[Theorem 4.1]{hlm} shows that $\rho=1$ in case we have $\eta_j=1$ for all $j$.
Integrating over $(\alp,\bet)\in \grN$ now delivers the asymptotic relation
$$\iint_\grN \mathscr F_\bfeta(\alpha,\beta)\,\mathrm d\alpha\,\mathrm d\beta 
=\rho\grS(Q)\mathfrak I(Q)+O(P^{s-8}(\log P)^{-1/4}),$$
and the desired conclusion follows by reference to Lemmata \ref{lemma3.1} and 
\ref{lemma3.3}.\par

Turning our attention now to the second asymptotic relation asserted in the statement, let 
$\grN^*$ denote the union of the rectangles $\grN(q,a,b)$ with
$$0\le a,b\le q,\quad 1\le q\le Q^2\quad \text{and}\quad (a,b,q)=1.$$
Then $\grN\subset \grM\times \grM \subset \grN^*$ and $\grN^*$ has measure 
$O(Q^8P^{-8})$. By \eqref{3.14}, we have
$$|\mathscr F_\bfeta(\alpha,\beta)|\ll \rho q^{-s}|T(q,a,b)V(\alpha-a/q,\beta-b/q)|+
P^s(\log P)^{-1/2}.$$
Thus, by integrating over $(\alp,\beta)\in \grN^*\setminus \grN$ and applying the bound 
\eqref{3.8} of Lemma \ref{lemma3.1} together with Lemma \ref{lemma3.2}, we obtain
$$\iint_{\grN^*\setminus \grN} |\mathscr F_\bfeta(\alpha,\beta)|\,\mathrm d\alpha
\,\mathrm d\beta\ll P^{s-8}\biggl(\sum_{Q<q\le Q^2}U^\dagger(q)+(\log P)^{-1/4}\biggr) 
\ll P^{s-8}Q^{\varepsilon-1}.$$
Since
$$\int_\grM\int_\grM \mathscr F_\bfeta(\alpha,\beta) \,\mathrm d\alpha\,\mathrm d\beta 
-\iint_\grN \mathscr F_\bfeta(\alpha,\beta) \,\mathrm d\alpha\,\mathrm d\beta \ll 
\iint_{\grN^*\setminus \grN} |\mathscr F_\bfeta(\alpha,\beta)|\,\mathrm d\alpha
\,\mathrm d\beta ,$$
the second asymptotic relation asserted in the statement now follows from the first.
\end{proof}

We close this section with a related auxiliary estimate for use in the next section.

\begin{lemma}\label{lemma3.5}
Suppose that $b$ is a non-zero integer and $0<\eta\le 1$. Then
$$\int_{\mathfrak M}|h(b\alpha;P,P^\eta)|^6\,\mathrm d\alpha\ll P^2.$$
\end{lemma}

\begin{proof} We abbreviate $h(\alpha;P,P^\eta)$ to $h(\alpha)$. Let $\alpha\in\grM(q,a)$ 
with $q\le Q$. Then, from \cite[Lemma 8.5]{Wsim2} one finds that there is number 
$\theta>0$ with
$$|h(b\alpha)|^6=\theta q^{-6}|S(q,ab)v(b\alpha-ab/q)|^6+O(P^6(\log P)^{-1/4}).$$
Since $\grM$ has measure $O(Q^3P^{-4})$, one sees that
$$\int_{\mathfrak M}|h(b\alpha)|^6\,\mathrm d\alpha \ll \sum_{q\le Q} 
\sum_{\substack{a=1 \\ (a,q)=1}}^q |S(q,ab)|^6\int_{-\infty}^\infty |v(\gamma)|^6
\,\mathrm d\gamma +P^2(\log P)^{-1/8}.$$
By \eqref{3.12}, the integral on the right hand side here is $O(P^2)$. Meanwhile, the 
remaining sum is readily estimated by an obvious adjustment in the proof of 
\cite[Lemma 4.9]{hlm}, as given for example in the proof of \cite[Lemma 5.1]{V89}. In this 
way, one finds that the sum on the right hand side is $O(1)$, and the conclusion of the 
lemma follows.
\end{proof}

\section{The proof of Theorem \ref{theorem1.2}}
In this section we choose $\eta$ in accordance with \eqref{2.1} and Lemma 
\ref{lemma2.1}, and we take $\eta_j=\eta$ for all $1\le j\le s$ in the generating function 
\eqref{3.3}. We consider a system of the shape (\ref{1.1}), and suppose that it is given in 
the form \eqref{1.6} in line with the hypotheses of Theorem \ref{theorem1.2}. In particular, 
we suppose that this system has non-singular solutions in each completion of the rational 
numbers.\par 

We continue to abbreviate $h(\alpha;P,P^\eta)$ to $h(\alpha)$ and put
$$F_1(\alpha)=h(a_1\alpha)\cdots h(a_6\alpha),\quad F_2(\beta)=h(b_1\beta)\cdots 
h(b_6\beta), $$
$$H(\alpha,\beta)=\prod_{i=7}^n h(a_i\alpha)\prod_{j=7}^m h(b_j\beta)
\prod_{k=1}^l h(c_k\alpha+d_k\beta).$$
Note that $H(\alp,\bet)$ is a product of $s-12$ Weyl sums, and that
\begin{equation}\label{4.1}
\mathscr F_\bfeta (\alpha,\beta)=F_1(\alpha)F_2(\beta)H(\alpha,\beta).
\end{equation}
Further, our assumptions currently in play allow us to combine the conclusions of Lemmata
\ref{lemma3.1}, \ref{lemma3.3} and \ref{lemma3.4} so as to infer that
\begin{equation}\label{4.2}
\int_\grM \int_\grM \mathscr F_\bfeta(\alpha,\beta)\,\mathrm d\alpha\,\mathrm d\beta \gg 
P^{s-8}.
\end{equation}

The argument that we now explain is based on the unorthodox treatment of certain cubic 
forms in our memoir \cite{BWCrelle}. Section 3 of that work serves as a blueprint, but in the 
current context the arithmetic mollifier is more radical than its precedessor so as to cope 
with smooth variables more easily.\par

Let $u$ and $v$ be integers, and let $\varrho(u,v)$ denote the number of solutions of the 
system of equations
\begin{align}
\sum_{i=7}^n a_ix_i^4+\sum_{k=1}^l c_kz_k^4&=u,\label{4.3}\\
\sum_{j=7}^m b_jy_j^4+\sum_{k=1}^l d_kz_k^4&=v,\label{4.4}
\end{align}
in integers
\begin{equation}\label{4.5}
x_i,y_j,z_k\in\mathscr A(P,P^\eta)\quad (7\le i\le n,\, 7\le j\le m,\, 1\le k\le l).
\end{equation}
Further, let $\varrho_1(u)$ denote the number of solutions of \eqref{4.3} satisfying 
\eqref{4.5}, and let $\varrho_2(v)$ denote the number of solutions of \eqref{4.4} satisfying 
\eqref{4.5}. Then
\begin{equation}\label{4.6}
\varrho_1(u)=\sum_{v\in\mathbb Z}\varrho(u,v)\quad \text{and}\quad 
\varrho_2(v)=\sum_{u\in\mathbb Z}\varrho(u,v).
\end{equation}
Note that $\varrho(u,v)=0$ unless $u$ and $v$ simultaneously satisfy the inequalities 
$|u|\le CP^4$, $|v|\le CP^4$, with $C=C(\bfa,\bfb,\bfc,\bfd)$ a sufficiently large constant.

\par Now let $L=\log\log P$ and $M=P^{s-16}L$. We partition $\dbZ^2$ into the three sets
\begin{align*}
\mathfrak X&=\{(u,v)\in\mathbb Z^2:\text{$\varrho_1(u)\le M$ and $\varrho_2(v)\le M$}
\},\\
\mathfrak Y_1&=\{(u,v)\in\mathbb Z^2:\text{$\varrho_1(u)> M$ and $\varrho_2(v)\le M$}
\},\\
\mathfrak Y_2&=\{(u,v)\in\mathbb Z^2:\varrho_2(v)> M\}.
\end{align*}
The next lemma shows that the sets $\mathfrak Y_1$ and $\mathfrak Y_2$ are slim.

\begin{lemma}\label{lemma4.1}
For $i=1$ and $2$ one has
$$\sum_{(u,v)\in\mathfrak Y_i}\varrho(u,v)\ll P^{s-12}L^{-1}.$$
\end{lemma}

\begin{proof} We begin by observing that, as a consequence of \eqref{4.6}, we have
\begin{equation}\label{4.7}
\sum_{(u,v)\in\mathfrak Y_1}\varrho(u,v)\le \sum_{\varrho_1(u)>M}\sum_{v\in \dbZ}
\varrho(u,v) \le M^{-1}\sum_{u\in \dbZ}\varrho_1(u)^2.
\end{equation}
Next, by orthogonality, we see that
$$\sum_{u\in \dbZ}\varrho_1(u)^2=\int_0^1 |H(\alpha,0)|^2\,\mathrm d\alpha.$$
Now $n-6+l$ of the variables \eqref{4.5} appear in \eqref{4.3} with a non-zero coefficient. 
The hypotheses of Theorem \ref{theorem1.2} ensure, moreover, that 
$$n-6+l\ge m+l-6\ge 6.$$
We apply \eqref{3.10} to exactly six of the respective factors in the product defining 
$H(\alp,\bet)$, and estimate the remaining Weyl sums trivially. Thus, on recalling (\ref{2.1}) 
and the conclusion of Lemma \ref{lemma2.1}, and applying a change of variables, we 
conclude that for some natural number $c$ one has
\begin{equation}\label{4.8}
\sum_{u\in \dbZ}\varrho_1(u)^2\ll P^{2(s-18)}\int_0^1|h(c\alpha)|^{12}\,\mathrm d\alpha
\ll P^{2s-28}.
\end{equation}
Hence, by means of (\ref{4.7}), one is led to the upper bound
$$\sum_{(u,v)\in\mathfrak Y_1}\varrho(u,v)\ll M^{-1}P^{2s-28}\ll P^{s-12}L^{-1}.$$
This confirms the case $i=1$ of the lemma, and a symmetrical argument delivers the case 
$i=2$.
\end{proof}

Let $\grA_1$ and $\grA_2$ be measurable subsets of $[0,1]$, and put
\begin{equation}\label{4.9}
R_j(w;\grA_j)=\int_{\grA_j}F_j(\alpha)e(\alpha w)\,\mathrm d\alpha\quad (j=1,2).
\end{equation}
The definition of $H(\alp,\bet)$ shows that
$$H(\alpha,\beta)=\sum_{(u,v)\in\mathbb Z^2}\varrho(u,v)e(\alpha u+\beta v),$$
and so it follows from (\ref{4.1}) that
\begin{equation}\label{4.10}
\int_{\grA_2}\int_{\grA_1}\mathscr F_\bfeta(\alpha,\beta)\,\mathrm d\alpha\,\mathrm 
d\beta =\sum_{(u,v)\in\mathbb Z^2}\varrho(u,v)R_1(u;\grA_1)R_2(v;\grA_2).
\end{equation}

Write
\begin{equation}\label{4.11}
N(\grA_1,\grA_2)=\sum_{(u,v)\in\grX}\varrho(u,v)R_1(u;\grA_1)R_2(v;\grA_2).
\end{equation}
Then, on comparing (\ref{3.4}) and (\ref{4.10}), it follows from \eqref{4.11} via 
orthogonality that $\mathscr N_\bfeta(P)\ge N([0,1],[0,1])$. Recalling the definitions of 
$\grM$ and $\grm$ from the preamble to Lemma \ref{lemma3.4}, we may therefore 
disassemble the unit square to obtain the lower bound
\begin{equation}\label{4.12}
\mathscr N_\bfeta(P)\ge N(\grM,\grM)+N(\grM,\grm)+N(\grm,\grM)+N(\grm,\grm).
\end{equation}

We now consider the terms on the right hand side of (\ref{4.12}) in turn.
 
\begin{lemma}\label{lemma4.2} One has
$$N(\grM,\grM)=\int_\grM \int_\grM \mathscr F_\bfeta(\alpha,\beta)\,\mathrm d\alpha\,
\mathrm d\beta+O(P^{s-8}L^{-1}).$$
\end{lemma}

\begin{proof} In \eqref{4.10} and \eqref{4.11} we take $\grA_1=\grA_2=\grM$. By 
considering the difference of these expressions, we see that
\begin{equation}\label{4.13}
N(\grM,\grM)=\int_\grM\int_\grM \mathscr F_\bfeta(\alp,\bet)\,\mathrm d\alp \,\mathrm 
d\bet -\mathscr E,
\end{equation}
where
$$\mathscr E=\sum_{(u,v)\in\mathfrak Y_1\cup\mathfrak Y_2}\varrho(u,v)
R_1(u;\grM)R_2(v;\grM).$$
By (\ref{3.10}), \eqref{4.9} and Lemma \ref{lemma3.5}, it is immediate that the upper 
bound
\begin{equation}\label{4.14}
R_j(w;\grM) \ll P^2 
\end{equation}
holds uniformly in $w\in\mathbb Z$ for $j=1$ and $2$. Thus, in view of Lemma 
\ref{lemma4.1}, we have
$$\mathscr E\ll P^4\sum_{(u,v)\in\mathfrak Y_1\cup\mathfrak Y_2}\varrho(u,v)\ll 
P^{s-8}L^{-1},$$
and the desired conclusion follows from (\ref{4.13}).\end{proof}

\begin{lemma}\label{lemma4.3}
One has
$$N(\grm,\grM)+N(\grM,\grm)\ll P^{s-8}L^{-1}.$$
\end{lemma}

\begin{proof} By taking $\grA_1=\grm$ and $\grA_2=\grM$ in \eqref{4.11}, we deduce via 
\eqref{4.14} that
$$N(\grm,\grM)\ll P^2\sum_{(u,v)\in\mathfrak X}\varrho(u,v)|R_1(u;\grm)|\ll P^2
\sum_{u\in\mathbb Z}\varrho_1(u) |R_1(u;\grm)|.$$
By \eqref{4.8} and Cauchy's inequality, we infer that
\begin{equation}\label{4.15}
N(\grm,\grM)\ll P^{s-12}\biggl( \sum_{u\in\mathbb Z}|R_1(u;\grm)|^2\biggr)^{1/2}.
\end{equation}
Meanwhile, from the argument of the proof of \cite[Lemma 8.1]{W03}, we find that for any 
fixed integer $c\neq 0$ one has $h(c\alpha)\ll P(\log P)^{-\tau^2}$ uniformly for 
$\alpha\in\grm$. Hence, by Lemma \ref{2.1},
\begin{equation}\label{4.16}
\int_\grm |h(c\alpha)|^{12}\,\mathrm d\alpha \ll P^{0.04}(\log P)^{-\tau^3}\int_0^1 
|h(c\alpha)|^{11.96}\,\mathrm d\alpha \ll P^8(\log P)^{-\tau^3}.  
\end{equation}
By Bessel's inequality, H\"older's inequality and \eqref{4.16}, we thus obtain the bound
\begin{equation}\label{4.17}
\sum_{u\in\mathbb Z}|R_1(u;\grm)|^2\le \int_\grm |F_1(\alpha)|^2\,\mathrm d\alpha \ll 
P^8(\log P)^{-\tau^3}.
\end{equation}
The bound $N(\grm,\grM)\ll P^{s-8}(\log P)^{-\tau^4}$ follows by substituting (\ref{4.17}) 
into (\ref{4.15}), and by symmetry, an upper bound of the same strength holds also for 
$N(\grM,\grm)$. The conclusion of the lemma is now immediate.
\end{proof}

\begin{lemma}\label{lemma4.4} One has
$$N(\grm,\grm)\ll P^{s-8}L^{-1}.$$
\end{lemma}

\begin{proof} Taking $\grA_1=\grA_2=\grm$ in \eqref{4.11}, we infer via (\ref{3.10}) that
$$N(\grm,\grm)\le \sum_{(u,v)\in\mathfrak X}\varrho(u,v)\left(|R_1(u;\grm)|^2+
|R_2(v;\grm)|^2\right).$$
But in view of (\ref{4.6}), one has
$$\sum_{(u,v)\in\mathfrak X}\varrho(u,v)|R_1(u;\grm)|^2\le \sum_{\varrho_1(u)\le M} 
\varrho_1(u)|R_1(u;\grm)|^2\le M\sum_{u\in\mathbb Z}|R_1(u;\grm)|^2. $$
We may now apply \eqref{4.17} and observe that a symmetrical argument provides an 
estimate of the same strength when $R_1$ is replaced by $R_2$. Since we have
$M\ll P^{s-16}(\log P)^\eps$, we conclude that
$$N(\grm,\grm)\ll P^{s-8}(\log P)^{\eps-\tau^3}\ll P^{s-8}L^{-1}.$$
This completes the proof of the lemma.
\end{proof}

Now suppose that the equations \eqref{1.6} have non-singular solutions in all completions 
of the rationals. Then, by \eqref{4.2} and Lemma \ref{lemma4.2} we see that 
$N(\grM,\grM)\gg P^{s-8}$. By \eqref{4.12} and Lemmata \ref{lemma4.3} and 
\ref{lemma4.4}, we have
$$\mathscr N_\bfeta(P)\ge N(\grM,\grM)+O(P^{s-8}L^{-1}).$$
Hence, in view of (\ref{3.5}), we conclude that $\mathscr N(P)\ge \mathscr N_\bfeta(P)
\gg P^{s-8}$. This completes the proof of Theorem \ref{theorem1.2}.

\section{The proof of Theorem \ref{theorem1.1}: large $n$}
In this section and the next we suppose that the hypotheses of Theorem \ref{theorem1.1} 
are satisfied, that the system \eqref{1.1} is given in the form \eqref{1.6}, and that $\eta$ 
is again chosen in accordance with \eqref{2.1} and Lemma \ref{lemma2.1}. Since all cases 
of Theorem \ref{theorem1.1} with $m\ge 6$ are covered by Theorem \ref{theorem1.2}, we 
may assume in addition that $m\le 5$. Our focus in the current section is the situation in 
which $n\ge 8$. We defer to the next section the corresponding scenario in which $n\le 7$.

\par Suppose then that $n\ge 8$. We take 
$$\eta_j=1\quad (1\le j\le 4)\quad \text{and}\quad \eta_j=\eta \quad (5\le j\le s).$$
The generating function $\mathscr F_\bfeta$ is then defined by means of the formula 
\eqref{3.3}. Provided that the equations \eqref{1.1} have non-singular solutions in all 
completions of the rational numbers, it now follows from Lemmata \ref{lemma3.1}, 
\ref{lemma3.3} and \ref{lemma3.4}  that
\begin{equation}\label{5.1}
\iint_\grN \mathscr F_\bfeta(\alpha,\beta)\,\mathrm d\alpha\,\mathrm d\beta \gg P^{s-8}.
\end{equation}
We imminently show that
\begin{equation}\label{5.2}
\iint_\grn \mathscr F_\bfeta(\alpha,\beta)\,\mathrm d\alpha\,\mathrm d\beta \ll  
P^{s-8}L^{-1}.
\end{equation}
The sum of the integrals in \eqref{5.1} and \eqref{5.2} is $\mathscr N_\bfeta(P)$. Thus, 
granted the validity of the upper bound (\ref{5.2}), we conclude from (\ref{3.5}) that 
$\mathscr N(P)\ge \mathscr N_\bfeta(P)\gg P^{s-8}$, so this will establish Theorem 
\ref{theorem1.1} for $n\ge 8$.\par

Before launching the proof of Theorem \ref{theorem1.1} for $n\ge 8$ in earnest, we 
prepare an auxiliary estimate of use in both this and the next section. In this context, we 
recall the notation introduced in (\ref{1.5}), and write
\begin{equation}\label{5.3}
f(\alpha)=h(\alpha;P,P)\quad \text{and}\quad h(\alpha)=h(\alpha;P,P^\eta). 
\end{equation}
Then, by \eqref{3.2} and \eqref{3.3}, we have
\begin{equation}\label{5.4}
\mathscr F_\bfeta(\alpha,\beta)=f(a_1\alpha)f(a_2\alpha)f(a_3\alpha)f(a_4\alpha)
\prod_{j=5}^s h(\Lambda_j).
\end{equation}

\begin{lemma}\label{lemma5.1} Suppose that $\Lambda_i$ and $\Lambda_j$ are linearly 
independent. Then, provided that $P$ is sufficiently large, and one has both 
$$|h(\Lambda_i)|\ge PL^{-300}\quad \text{and}\quad |h(\Lambda_j)|\ge PL^{-300},$$
it follows that $(\alpha,\beta)\in\grN +\mathbb Z^2$.
\end{lemma}

\begin{proof} Let $\nu\in\{i,j\}$ and suppose that $P$ is sufficiently large. From the 
argument of the proof of \cite[Lemma 8.1]{W03} it follows that there exist 
$a_\nu\in\mathbb Z$ and $q_\nu\in\mathbb N$ with $q_\nu\le L^{3000}$ and 
$|\Lambda_\nu-a_\nu/q_\nu|\le L^{3000}P^{-4}$. By eliminating $\beta$, or $\alpha$, 
respectively, from the two inequalities here, we see that there exist $a,b\in\mathbb Z$ and 
$q\in\mathbb N$ with $(q,a,b)=1$ and $q\le L^{6001}$ satisfying the property that
$$|\alpha -a/q|\le L^{3001}P^{-4}\quad \text{and}\quad |\beta -b/q|\le L^{3001}P^{-4}.
$$
Thus, on noting that $L^{6001}\le (\log P)^\tau$ when $P$ is sufficiently large, we 
conclude that $(\alp,\bet)\in \grN+\dbZ^2$. This completes the proof of the lemma.
\end{proof} 

It is now the moment to launch the proof of Theorem \ref{theorem1.1} in the case 
$n\ge 8$. Consider the linear forms $\Lambda_j$ with $s-17\le j\le s$, presented 
simultaneously in the formats \eqref{1.5} and \eqref{3.2}. Then the hypothesis 
$q_0\ge 12$ implies that $B_j\neq 0$ for $s-11\le j\le s$. Hence, amongst the indices $j$ 
with $s-17\le j\le s$, we see that the linear form $\Lam_j$ can be independent of $\beta$ 
no more than six times. Restricting attention to the forms that depend explicitly on $\beta$, 
the number of repetitions amongst the numbers $A_j/B_j$ is at most $m\le 5$. It follows 
that the $18$ forms $\Lambda_j$ with $s-17\le j\le s$ can be grouped into six disjoint 
subsets $\{\Lambda_\kappa,\Lambda_\mu,\Lambda_\nu\}$ satisfying the condition that 
$\Lam_\kappa$, $\Lam_\mu$ and $\Lam_\nu$ are pairwise linearly independent. By 
\eqref{3.10} with $r=6$ we then see that
\begin{equation}\label{5.5}
\int_0^1\!\!\int_0^1\prod_{j=s-17}^s |h(\Lambda_j)|\,\mathrm d\alpha\, \mathrm d\beta 
\ll \sideset{}{'}\sum_{\kappa<\mu<\nu}\int_0^1\!\!\int_0^1|h(\Lambda_\kappa)
h(\Lambda_\mu)h(\Lambda_\nu)|^6 \,\mathrm d\alpha\,\mathrm d\beta ,
\end{equation} 
where the sum is over all triples $\kappa,\mu,\nu$ satisfying the condition that 
$\Lambda_\kappa,\Lambda_\mu,\Lambda_\nu$ are pairwise linearly independent. By 
Theorem \ref{theorem2.4}, it follows that the integral on the right hand side of \eqref{5.5} 
is $O(P^{21/2-2\tau})$. But $s-17\ge 5$ so trivial estimates for the remaining exponential 
sums suffice to conclude that
\begin{equation}\label{5.6}
\int_0^1\!\!\int_0^1 \prod_{j=5}^s |h(\Lambda_j)|\,\mathrm d\alpha\,\mathrm 
d\beta \ll P^{s-12-2\tau +1/2}. 
\end{equation}

\par We apply the estimate (\ref{5.6}) first to reduce the integration over $\alp$ to a broad 
set of major arcs. Let $\grK$ denote the union of the intervals 
$$\grK(q,a)=\{\alpha\in[0,1]: |q\alpha-a|\le P^{-3}\},$$
with $0\le a\le q\le P$ and $(a,q)=1$. Also, let $\mathfrak k=[0,1]\setminus \grK$. By 
Weyl's inequality (see \cite[Lemma 2.4]{hlm}) we see that the bound 
$f(a_j\alpha)\ll P^{7/8+\varepsilon}$ $(1\le j\le 4)$ holds uniformly for 
$\alpha\in\mathfrak k$. Hence, by \eqref{5.6}, we discern that
\begin{equation}\label{5.7}
\int_0^1\!\!\int_{\mathfrak k}|\mathscr F_\bfeta(\alpha,\beta)|\, \mathrm d\alpha\,
\mathrm d\beta \ll P^{7/2+\varepsilon}\cdot P^{s-12-2\tau +1/2}\ll P^{s-8-\tau}.
\end{equation}

\par Next we seek to prune the remaining domain $\grK\times [0,1]$ down to the narrow 
set of major arcs $\grN$. For $5\le k\le s$, let
$$\grn_k=\{(\alpha,\beta)\in\grK\times[0,1]: |h(\Lambda_k)|\le PL^{-300}\}.$$
Also, put
$$J_k=\iint_{\grn_k}|\mathscr F_\bfeta(\alpha,\beta)|\, \mathrm d\alpha\,\mathrm d\beta
\quad (5\le k\le s).$$
Suppose that $s-11\le k\le s$, so that $B_k\ne 0$. In order to initiate the estimation of 
$J_k$, we temporarily write
$$F(\alpha)=|f(a_1\alpha)\cdots f(a_4\alpha)|\quad \text{and}\quad 
G(\alpha)=|h(a_5\alpha)\cdots h(a_8 \alpha)|.$$
Then, since in this section we work under the assumption that $n\ge 8$, we deduce from 
\eqref{5.4} that whenever $(\alpha,\beta)\in\grn_k$, one has
\begin{align*}
|\mathscr F_\bfeta(\alpha,\beta)|&\le P^{s-20}F(\alpha)G(\alpha)\prod_{j=s-11}^s
|h(\Lambda_j)|\\
&\le P^{s-20}F(\alpha)G(\alpha)\left( P^{12}L^{-300}\right)^{1/300}\prod_{j=s-11}^s 
|h(\Lambda_j)|^{299/300}. 
\end{align*}
By integrating over $(\alp,\bet)\in \grn_k$, we deduce that
\begin{equation}\label{5.8}
J_k\le P^{s-20+1/25}L^{-1}\int_\grK F(\alpha)G(\alpha)\int_0^1\prod_{j=s-11}^s
|h(\Lambda_j)|^{299/300}\,\mathrm d\beta\,\mathrm d \alpha.
\end{equation}

\par In the inner integral on the right hand side of (\ref{5.8}), we apply H\"older's inequality 
to obtain the bound
\begin{equation}\label{5.8a}
\int_0^1\prod_{j=s-11}^s|h(\Lambda_j)|^{299/300}\,\mathrm d\beta \le 
\prod_{j=s-11}^s\biggl( \int_0^1 |h(\Lambda_j)|^{11.96}\,\mathrm d\beta \biggr)^{1/12}.
\end{equation}
Recall that $B_j\neq 0$ for $s-11\le j\le s$. In particular, the exponential sum 
$h(\Lambda_j)$ has period $1/|B_j|$ in $\beta$, and thus
$$\int_0^1|h(\Lambda_j)|^{11.96}\,\mathrm d\beta =\int_0^1|h(B_j\beta)|^{11.96}\,
\mathrm d\beta = \int_0^1|h(\beta)|^{11.96}\,\mathrm d\beta.$$
Consequently, an application of Lemma \ref{lemma2.1} reveals that
$$\int_0^1|h(\Lambda_j)|^{11.96}\,\mathrm d\beta \ll P^{7.96}.$$
On substituting this estimate into (\ref{5.8a}), we obtain a bound for the inner integral in 
\eqref{5.8} that is independent of $\alpha$. Hence we arrive at the relation
\begin{equation}\label{5.9}
J_k\ll P^{s-12}L^{-1}\int_\grK F(\alpha)G(\alpha)\,\mathrm d\alpha.
\end{equation}
  
\par It remains to estimate the integral on the right hand side of (\ref{5.9}). By an 
argument largely identical to the one that yielded Lemma \ref{lemma3.5}, but this time 
invoking \cite[Theorem 4.1]{hlm}, one readily confirms the estimate  
\begin{equation}\label{5.10}
\int_\grK |f(a\alpha)|^6\,\mathrm d\alp \ll P^2
\end{equation}
that is valid for any fixed choice of $a\in\mathbb Z\setminus\{0\}$. Hence, by H\"older's 
inequality,
$$\int_\grK F(\alpha)^{3/2}\,\mathrm d \alpha \ll P^2.$$
By H\"older's inequality again, Lemma \ref{lemma2.1} delivers the bound
$$\int_0^1 G(\alpha)^3\,\mathrm d\alpha \le \prod_{j=5}^8\biggl( \int_0^1
|h(a_j\alpha)|^{12}\,\mathrm d\alpha \biggr)^{1/4}\ll P^8.$$
Another application of H\"older's inequality therefore leads us to the estimate
$$\int_\grK F(\alpha)G(\alpha)\,\mathrm d\alpha \ll \biggl( \int_\grK F(\alp)^{3/2}
\,\mathrm d\alp \biggr)^{2/3}\biggl( \int_0^1G(\alp)^3\,\mathrm d\alp \biggr)^{1/3}\ll 
P^4.$$
By substituting this estimate into (\ref{5.9}) and recalling the definition of $J_k$, we obtain 
the bound
\begin{equation}\label{5.11}
\iint_{\mathfrak n_k}|\mathscr F_\bfeta(\alpha,\beta)|\, \mathrm d\alpha\,\mathrm d\beta 
\ll P^{s-8}L^{-1}\quad (s-11\le k\le s). 
\end{equation}

Now suppose that $(\alpha,\beta)\in[0,1]^2$ is neither in $\mathfrak k\times [0,1]$ nor in 
one of the sets $\grn_k$ with $s-11\le k\le s$. Then $|h(\Lambda_k)|\ge PL^{-300}$ for 
$s-11\le k\le s$. But $m\le 5$, so among the forms $\Lambda_j$ with $s-11\le j\le s$ we 
can find a pair of linearly independent forms. By Lemma \ref{lemma5.1}, this implies that 
$(\alpha,\beta)\in\grN+\dbZ^2$. We therefore have the relation
$$\grn\subseteq \left( \mathfrak k\times [0,1]\right) \cup \grn_{s-11}\cup \ldots \cup 
\grn_s,$$
and consequently it follows from \eqref{5.7} and \eqref{5.11} that
\begin{align*}
\iint_{\mathfrak n}|\mathscr F_\bfeta(\alpha,\beta)|\, \mathrm d\alpha\,\mathrm d\beta 
&\le \int_0^1\!\!\int_{\mathfrak k}|\mathscr F_\bfeta(\alpha,\beta)|\, \mathrm d\alpha\,
\mathrm d\beta +\sum_{k=s-11}^s\iint_{\mathfrak n_k}|\mathscr F_\bfeta(\alpha,\beta)|\, 
\mathrm d\alpha\,\mathrm d\beta \\
&\ll P^{s-8}L^{-1}. 
\end{align*}
This confirms \eqref{5.2} and completes the proof of Theorem \ref{theorem1.1} when 
$n\ge 8$.\par

We briefly pause to indicate the relevance of our estimates from \cite{BWnew} to the 
progress in this section. This concerns the cases where $q_0=12$, in which case $n\ge 10$. 
In such circumstances exactly $12$ Weyl sums $h(\Lambda_j)$ depend on $\beta$, and 
some minor arc information has to be extracted from the $\beta$-integration. This would be 
straightforward were one or two of the relevant Weyl sums to have smoothness parameter 
$\eta_j$ equal to $1$, but this would be in conflict with the requirements in \eqref{5.5} 
where all Weyl sums must be smooth. Thus we are fortunate to have optimal control on a 
moment smaller than the twelfth of a smooth biquadratic Weyl sum.

\section{The proof of Theorem \ref{theorem1.1}: small $n$}
We devote this section to a discussion of the situation in which $n\le 7$. Our previous work 
\cite{JEMS1} in fact handles all cases in which $n\le 7$. However, equipped with the new 
mean value estimate provided by Theorem \ref{theorem2.4}, we are able now to offer a 
streamlined treatment. We include the present discussion of the situation with $n\le 7$, 
therefore, in order that our account be self-contained. This treatment may also offer 
inspiration for future investigations concerning the solubility of systems of diagonal 
equations. We now have 
$$r_1=n\le 7,\quad r_3\le r_2=m\le 5\quad \text{and}\quad s\ge 22,$$
whence necessarily $r_4\ge 1$. For $1\le \nu\le 4$ let $i_\nu \in \{1,2,\ldots,s\}$ be any 
index counted by $r_\nu$. Thus, for example, we may take $i_1=1$ and $i_2=n+1$. Let
$$\mathrm I=\{1,2,\ldots,s\}\setminus \{i_1,i_2,i_3,i_4\}.$$
It transpires that a combinatorial observation concerning the set $\mathrm I$ greatly 
facilitates our analysis of certain auxiliary mean values, and this we formulate in the next 
lemma. In this context, we recall the notion of equivalence of indices introduced following 
equation (\ref{1.5}).

\begin{lemma}\label{lemma6.1} Subject to the hypotheses in the preamble, fix indices 
$\nu$ with $1\le \nu\le 4$ and $k\in\mathrm I$, and suppose that the linear forms 
$\Lambda_{i_\nu}$ and $\Lambda_k$ are linearly independent. Then $\mathrm I$ has two 
disjoint subsets $\mathrm I_0$ and $\mathrm I_1$, with eight elements each, and 
possessing the following properties:
\begin{enumerate}
\item[(i)] one has $k\in \mathrm I_0$, and for all $j\in\mathrm I_0$ the forms $\Lambda_j$ 
and $\Lambda_{i_\nu}$ are linearly independent;
\item[(ii)] the maximal number of mutually equivalent indices $j\in\mathrm I_1$ is four.
\end{enumerate}
\end{lemma}

\begin{proof} First suppose that $\nu=1$. Then there are $r_1-1$ indices $j\in\mathrm I$ 
that are counted by $r_1$, and $r_1=n\le 7$. We put $\min \{4,r_1-1\}$ of these indices 
into the set $\mathrm I_1$ and discard the remaining ones (if any). Note that we discard at 
most two indices in this first phase of the assignment process.\par

If $r_2=m=5$ then four indices $j\in\mathrm I$ are counted by $r_2$. If the index $k$ is 
among them, then we put it into the set $\mathrm I_0$ and put the remaining three indices 
into the set $\mathrm I_1$. If the index $k$ is not among them, then we put all four 
indices into the set $\mathrm I_1$, and then place the index $k$ into the set 
$\mathrm I_0$. Under the current assumptions, we have $r_1\ge r_2=5$, so following 
these two rounds of allocations, the subset $\mathrm I_1$ already has $7$ or $8$ 
elements. Since we discarded at most two indices earlier, we have sufficient indices not 
already allocated from $\mathrm I$ that we may assign the remaining ones to 
$\mathrm I_0$ and $\mathrm I_1$ arbitrarily so that each subset emerges with eight 
elements apiece. It is then apparent that properties (i) and (ii) are satisfied for this 
assignment.\par

If, meanwhile, one has $r_2=m\le 4$, then we adjust the second phase of the assignment 
process by putting the index $k$ into the set $\mathrm I_0$, but otherwise we assign 
indices not already allocated from $\mathrm I$ to $\mathrm I_0$ and $\mathrm I_1$ 
arbitrarily so that each subset again emerges with eight elements apiece. It is then again 
apparent that properties (i) and (ii) are satisfied for this assignment. This completes the 
proof of the lemma when $\nu=1$.\par

Now suppose that $\nu\ge 2$. Then there are $r_\nu -1\le r_2-1=m-1\le 4$ indices 
$j\in\mathrm I$ that are counted by $r_\nu$, and these we insert into the set 
$\mathrm I_1$. All of the $r_1-1$ indices $j\in\mathrm I$ counted by $r_1$ we put into the 
set $\mathrm I_0$. If the index $k$ is not yet in $\mathrm I_0$ then we insert it into this 
set. So far, we have at most $(r_1-1)+1\le 7$ elements assigned to $\mathrm I_0$ and at 
most $4$ elements assigned to $\mathrm I_1$.\par

If there is a suffix $\kappa\not \in \{1,\nu\}$ for which $r_\kappa=5$ and the index $k$ is 
not counted by $r_\kappa$, then all five indices $j\in\mathrm I$ counted by $r_\kappa$ are 
not yet distributed, and we put one into $\mathrm I_0$ and four into $\mathrm I_1$. Both 
in this situation, and when there is no such suffix $\kappa$, we may assign the remaining 
indices not already allocated from $\mathrm I$ to $\mathrm I_0$ and $\mathrm I_1$ 
arbitrarily so that each subset once more emerges with eight elements apiece. On this 
occasion as in earlier cases, it is again apparent that properties (i) and (ii) are satisfied for 
this assignment. This completes the proof of the lemma when $\nu\ge 2$.
\end{proof}

We now embark on the proof of Theorem \ref{theorem1.1} when $n\le 7$. We take
$$\eta_j=\eta \quad (j\in \mathrm I)\quad \text{and}\quad \eta_{i_\nu}=1\quad 
(1\le \nu\le 4).$$
We then define $\mathscr F_\bfeta$ by means of \eqref{3.3}. Again adopting the notation 
introduced in (\ref{5.3}), and writing
$$f_\nu=f(\Lambda_{i_\nu})\quad (1\le \nu\le 4)\quad \text{and}\quad 
H_\Tet=\prod_{j\in\Tet}h(\Lambda_j)\quad (\Tet\subseteq \mathrm I),$$
we have $\mathscr F_\bfeta = f_1f_2f_3f_4 H_{\mathrm I}$. It then follows as before that  
the proof of Theorem \ref{theorem1.1} when $n\le 7$ is completed by confirming the upper 
bound \eqref{5.2}.\par

Let $1\le \nu\le 4$, and let $\mathfrak a\subset [0,1]^2$ be measurable. Put
\begin{equation}\label{6.1}
I_\nu(\mathfrak a)=\iint_{\mathfrak a}|f_\nu|^4|H_{\mathrm I}|\,\mathrm d\alpha\,
\mathrm d\beta. 
\end{equation}
Then, an application of \eqref{3.10} reveals that
\begin{equation}\label{6.2}
\iint_\grn |\mathscr F_\bfeta|\,\mathrm d\alpha\,\mathrm d\beta \ll I_1(\grn)+I_2(\grn)
+I_3(\grn)+I_4(\grn).
\end{equation}
We fix $\nu\in\{1,2,3,4\}$ and estimate $I_\nu(\grn)$ by a method that is similar to that 
applied in the situation with $n\ge 8$ analysed in the previous section. We begin by defining 
the sets
\begin{align*}
\mathfrak p&=\{(\alpha,\beta)\in [0,1]^2:\Lambda_{i_\nu}\in\mathfrak k+\mathbb Z\},\\ 
\mathfrak P&=\{(\alpha,\beta)\in [0,1]^2:\Lambda_{i_\nu}\in\mathfrak K+\mathbb Z\} 
\end{align*}
and
$$\mathfrak P_k=\{(\alpha,\beta)\in\mathfrak P:|h(\Lambda_k)|\le PL^{-300}\}\quad 
(k\in\mathrm I).$$
Observe that $\text{card}(\mathrm I)\ge 18$. Moreover, under the current hypotheses, at 
most six of the forms $\Lambda_j$ with $j\in\mathrm I$ can be multiples of one another. It 
follows that there are indices $k_1$ and $k_2$ in $\mathrm I$ having the property that the 
linear forms $\Lambda_{k_1}$, $\Lambda_{k_2}$ and $\Lambda_{i_\nu}$ are pairwise 
linearly independent. Hence, if $(\alpha,\beta)\in[0,1]^2$ is neither in $\mathfrak p$ nor in 
$\mathfrak P_{k_1}\cup \mathfrak P_{k_2}$, then $|h(\Lambda_{k_\iota})|>PL^{-300}$ 
for both $\iota= 1$ and $\iota=2$, and hence Lemma \ref{lemma5.1} implies that 
$(\alpha,\beta)\in\grN$. We therefore deduce that for these indices $k_1$ and $k_2$, we 
have $\grn\subseteq \grp\cup \grP_{k_1}\cup \grP_{k_2}$. Consequently, the upper bound 
(\ref{5.2}) follows from (\ref{6.2}) once we confirm that for $1\le \nu\le 4$, and all 
$k\in \mathrm I$ for which $\Lam_k$ is independent of $\Lam_{i_\nu}$, one has the 
estimates
\begin{equation}\label{6.3}
I_\nu(\grp)\ll P^{s-8}L^{-1}\quad \text{and}\quad I_\nu(\grP_k)\ll P^{s-8}L^{-1}.
\end{equation}

\par We begin by estimating $I_\nu(\mathfrak p)$. Since $i_1$ is an index counted by 
$r_1$, the current assumptions show that among the indices $j\in\mathrm I$, no more than 
six are mutually equivalent. Hence, an appropriate reinterpretation of the argument leading 
to \eqref{5.6} yields the bound
$$\int_0^1\!\!\int_0^1 |H_{\mathrm I}(\alpha,\beta)|\,\mathrm d\alpha\,\mathrm d\beta  
\ll P^{s-12-2\tau +1/2}.$$
Further, Weyl's inequality \cite[Lemma 2.4]{hlm} shows that $f_\nu\ll P^{7/8+\varepsilon}$ 
for $(\alpha,\beta)\in\mathfrak p$. Hence, as an echo of \eqref{5.7}, we deduce from 
\eqref{6.1} that
\begin{equation}\label{6.4}
I_\nu(\mathfrak p)\ll P^{s-8-\tau}.
\end{equation}

\par The treatment of the sets $\mathfrak P_k$ makes use of Lemma \ref{lemma6.1}. In 
the notation of this lemma, we have 
$|H_{\mathrm I}|\le P^{s-20}|H_{\mathrm I_0}H_{\mathrm I_1}|$. Thus, by applying 
H\"older's inequality to (\ref{6.1}), we deduce that
\begin{equation}\label{6.5}
I_\nu(\mathfrak P_k)\le P^{s-20}T_1^{1/3}T_2^{2/3},
\end{equation}
where
\begin{equation}\label{6.6}
T_1=\int_0^1\!\!\int_0^1 |H_{\mathrm I_1}|^3\,\mathrm d\alpha\,\mathrm d\beta 
\quad \text{and}\quad T_2=\iint_{\mathfrak P_k}|f_\nu|^6|H_{\mathrm I_0}|^{3/2}
\,\mathrm d\alpha\,\mathrm d\beta .
\end{equation}
It follows from the property (ii) associated with the set $\mathrm I_1$ in Lemma 
\ref{lemma6.1} that this set contains four disjoint subsets $\{i,j\}$ having the property that 
$\Lambda_i$ and $\Lambda_j$ are linearly independent. We therefore see from 
\eqref{3.10} with $r=4$ that
\begin{equation}\label{6.7}
|H_{\mathrm I_1}(\alpha,\beta)|^3\ll \sum |h(\Lambda_i)h(\Lambda_j)|^{12},
\end{equation}
where the summation is taken over the four pairs $i,j$ comprising these subsets. A change 
of variables in combination with orthogonality and Lemma \ref{lemma2.1} deliver the bound
$$\int_0^1\!\!\int_0^1 |h(\Lambda_i)h(\Lambda_j)|^{12}\,\mathrm d\alpha\,\mathrm 
d\beta =\int_0^1\!\!\int_0^1 |h(\alpha)h(\beta)|^{12}\,\mathrm d\alpha\,\mathrm d\beta 
\ll P^{16}$$
for each of the four pairs $i,j$. We therefore deduce from (\ref{6.6}) and (\ref{6.7}) that
\begin{equation}\label{6.8} 
T_1\ll P^{16}.
\end{equation}

Next we estimate the mean value $T_2$. By applying H\"older's inequality in (\ref{6.6}), we 
obtain the bound
\begin{equation}\label{6.9}
T_2\le \prod_{j\in\mathrm I_0}\biggl( \iint_{\mathfrak P_k}|f_\nu|^6|h(\Lambda_j)|^{12}
\,\mathrm d\alpha\,\mathrm d\beta\biggr)^{1/8}.
\end{equation}
It follows from the property (i) associated with the set $\mathrm I_0$ in Lemma 
\ref{lemma6.1} that the forms $\Lambda_{i_\nu}$ and $\Lambda_j$ are linearly 
independent when $j\in \mathrm I_0$. By the transformation formula applied to the 
non-singular linear map $(\alpha,\beta)\to (\Lambda_{i_\nu}, \Lambda_j)$, it follows via 
\eqref{5.10} and Lemma \ref{lemma2.1} that
\begin{equation}\label{6.10}
\iint_{\mathfrak P_k}|f_\nu|^6|h(\Lambda_j)|^{12}\,\mathrm d\alpha\,\mathrm d\beta\ll 
\int_{\mathfrak K}|f(\xi)|^6\,\mathrm d\xi \int_0^1|h(\zeta)|^{12}\,\mathrm d\zeta \ll 
P^{10}.
\end{equation}
Here we made use of the fact that $\mathfrak P_k$ maps to a compact set, and in addition 
that $f$ and $h$ are functions of period $1$. It follows from the property (i) associated with 
the set $\mathrm I_0$ in Lemma \ref{lemma6.1}, moreover, that $k\in\mathrm I_0$, and in 
this case the last estimate can be improved. Since one has the upper bound 
$|h(\Lambda_k)|\le PL^{-300}$ for $(\alpha,\beta)\in\mathfrak P_k$, we see that
$$\iint_{\mathfrak P_k}|f_\nu|^6|h(\Lambda_k)|^{12}\,\mathrm d\alpha\,\mathrm d\beta
\le P^{0.04}L^{-12}\iint_{\mathfrak P_k}|f_\nu|^6|h(\Lambda_k)|^{11.96}\,\mathrm 
d\alpha\,\mathrm d\beta.$$
We may now apply the transformation formula as before, and conclude via Lemma 
\ref{lemma2.1} that
\begin{equation}\label{6.11}
\iint_{\mathfrak P_k}|f_\nu|^6|h(\Lambda_k)|^{12}\,\mathrm d\alpha\,\mathrm d\beta 
\ll P^{10}L^{-12}.
\end{equation}
By substituting (\ref{6.10}) and (\ref{6.11}) into (\ref{6.9}), we arrive at the upper bound
\begin{equation}\label{6.12}
T_2\ll P^{10}L^{-3/2}.
\end{equation}

We are now equipped to derive the conclusion we have sought. By substituting (\ref{6.8}) 
and (\ref{6.12}) into (\ref{6.5}), we infer the estimate
$$I_\nu(\mathfrak P_k)\ll P^{s-20}(P^{16})^{1/3}(P^{10}L^{-3/2})^{2/3}
=P^{s-8}L^{-1}$$
that is valid for all $k\in\mathrm I$ satisfying the condition that $\Lambda_k$ is linearly 
independent of $\Lambda_{i_\nu}$. This, together with (\ref{6.4}), confirms the bounds 
(\ref{6.3}). Hence, as we explained in the preamble to (\ref{6.3}), the upper bound 
\eqref{5.2} does indeed hold. Since the bound (\ref{5.1}) holds in the present 
circumstances, once more as a consequence of Lemmata \ref{lemma3.1}, \ref{lemma3.3} 
and \ref{lemma3.4}, we again conclude from (\ref{3.5}) that 
$\mathscr N(P)\ge \mathscr N_\bfeta (P)\gg P^{s-8}$. This completes the proof of 
Theorem \ref{theorem1.1} for $n\le 7$.

\section{Obstructions}
In this section we discuss the pair of equations \eqref{1.2} and substantiate our 
introductory claims concerning some of its properties. Bright \cite{MB} has analysed in detail 
the validity of the Hasse principle for diagonal quartic forms in four variables. The following 
example (see \cite[Example 2.3]{MB}), is just one of the cases where the Hasse principle 
fails.

\begin{lemma}\label{lemma7.1} Let $p$ be a prime. Then the equation 
$x_1^4+x_2^4-6x_3^4-12x_4^4=0$ has a non-trivial solution in $\mathbb Q_p$, but only 
the trivial solution in $\mathbb Z$. 
\end{lemma}     

We now study the form $y_1^4+y_2^4+3y_3^4+5y_4^4+7y_5^4$. It turns out that this 
form is universal over all $\mathbb Q_p$, in a strong sense. This is the content of Lemma 
\ref{lemma7.4} below. We begin with the distribution of the above form in residue classes.

\begin{lemma}\label{lemma7.2} Let $p$ be an odd prime, and let $a\in \mathbb Z$. Then, 
the congruence
\begin{equation}\label{7.1}
y_1^4+y_2^4+3y_3^4+5y_4^4+7y_5^4 \equiv a \mmod p
\end{equation}
has a solution with not all the variables divisible by $p$. If $p=3$, $5$ or $7$, then the 
solution can be so chosen that the variable with coefficient $p$ is zero.
\end{lemma}

\begin{proof} The set $\{y^4: 1\le y\le p-1\}$ has at least $(p-1)/4$ elements. We write 
\begin{equation}\label{7.2}
b_1=b_2=1,\quad b_3=3,\quad b_4=5,\quad b_5=7.
\end{equation}
Then, by repeated use of the Cauchy-Davenport Theorem \cite[Lemma 2.14]{hlm}, for 
$p>7$ and $1\le \nu\le 5$, one finds that the set
$$\biggl\{ \sum_{j=1}^\nu b_jy_j^4:\text{$1\le y_1\le p-1$ and $1\le y_j\le p$ 
$(2\le j\le \nu)$}\biggr\}$$
contains at least $\min\{\nu(p-1)/4, p\}$ residue classes modulo $p$. Consequently, when 
$p\ge 11$  and $\nu=5$, all congruence classes modulo $p$ are covered, and in particular 
the class $a$ modulo $p$. This proves the lemma when $p\ge 11$.\par

If $p=7$ then there are $3$ fourth power residues modulo 7, and the above argument 
shows that when $1\le \nu\le 3$, the above set contains at least $\min\{3\nu,7\}$ residue 
classes modulo $7$. In particular, the values of the form $y_1^4+y_2^4+3y_3^4$, with 
$1\le y_1\le 6$ and $1\le y_2,y_3\le 7$, range over all residue classes modulo $7$, and 
in particular the class $a$ modulo $7$. In the cases $p=3$ and $5$ one has $y_i^4\in 
\{0,1\}$ modulo $p$, and a trivial check of cases verifies that all residue classes $a$ 
modulo $p$ are indeed covered in the prescribed manner.
\end{proof}

\begin{lemma}\label{lemma7.3} Let $a\in\mathbb Z$. Then the congruence
\begin{equation}\label{7.3}
y_1^4+y_2^4+3y_3^4+5y_4^4+7y_5^4\equiv a\mmod{16}
\end{equation}
has a solution in which one at least of the variables is equal to $1$.
\end{lemma}

\begin{proof} One has $y_i^4\in\{0,1\}$ modulo $16$, and so a direct check of cases 
verifies that all residues $a$ modulo $16$ possess a representation of the desired type.
\end{proof}

\begin{lemma}\label{lemma7.4} Let $p$ be a prime and $a\in \mathbb Z$. Then the 
equation
$$z_1^4+z_2^4+3z_3^4+5z_4^4+7z_5^4=a$$
has a solution in $\mathbb Z_p$ in which not all of the variables are zero.
\end{lemma}

\begin{proof} Let $p$ be an odd prime and adopt the notation (\ref{7.2}). Then it follows 
from Lemma \ref{lemma7.2} that there is a solution $\bfy\in \dbZ^5$ of the congruence 
\eqref{7.1} having the property that for some index $j$ one has $p\nmid 4b_jy_j^3$. Fixing 
$z_i=y_i$ for $i\ne j$, it is a consequence of Hensel's lemma that there exists 
$z_j\in\mathbb Z_p\setminus \{0\}$ with $z_j\equiv y_j\mmod p$ and 
$b_1z_1^4+\ldots +b_5z_5^4=a$. Thus, indeed, one has $\bfz\ne \mathbf0$.\par

When $p=2$ only modest adjustments are required in this argument. Here one observes 
that the structure of the group of reduced residues modulo $2^h$ for $h\ge 3$ ensures 
that if $b$ is an integer with $b\equiv 1\mmod{16}$, then there is a $2$-adic integer 
$\beta$ with $\bet^4=b$. Since Lemma \ref{lemma7.3} ensures that the congruence 
(\ref{7.3}) has a solution $\bfy\in \dbZ^5$ having the property that for some index $j$ one 
has $y_j=1$, we may proceed as before to show that there is a solution $\bfz\in 
\dbZ_2^5\setminus \{\mathbf 0\}$ to the equation $b_1z_1^4+\ldots +b_5z_5^4=a$.
\end{proof}

Now we turn to the system of equations \eqref{1.2}. Let $p$ be any prime number. It 
follows from Lemma \ref{lemma7.1} that there is a solution $(x_1,\ldots ,x_4)\in 
\dbZ_p^4\setminus \{\mathbf 0\}$ of the equation $x_1^4+x_2^4-6x_3^4-12x_4^4=0$. 
By multiplying the coordinates of this solution by an appropriate unit, moreover, there is no 
loss of generality in assuming that $x_4$ is a rational integer. Fixing this value of $x_4$, 
there is a solution $(x_5,\ldots ,x_9)\in \mathbb Z_p^5\setminus \{\mathbf 0\}$ of the 
equation
$$ x_4^4=7x_5^4+5x_6^4+3x_7^4+x_8^4+x_9^4.$$
This follows from Lemma \ref{lemma7.4}. With $x_j=0$ for $j\ge 10$, this provides us with 
a non-singular solution of \eqref{1.2} with coordinates in $\mathbb Z_p$. The existence of 
non-singular real solutions to the system (\ref{1.2}) is plain. This confirms that the system 
of equations (\ref{1.2}) has non-singular solutions in all completions of the rationals, as 
claimed in the preamble to the statement of Theorem \ref{theorem1.1}.\par

However, as a consequence of Lemma \ref{lemma7.1}, any solution of \eqref{1.2} in 
rational integers must have $x_1=x_2=x_3=x_4=0$. But then one has
$$7x_5^4+5x_6^4+3x_7^4+x_8^4+\ldots +x_s^4=0.$$
Since the form on the left hand side of this equation is positive definite, we are forced to 
conclude that $x_j=0$ for $1\le j\le s$, whence \eqref{1.2} has no solution in the integers 
other than the trivial solution $\bfx=\mathbf0$. As claimed in the introduction, this shows 
that the Hasse principle fails for the system of equations \eqref{1.2}. 

\bibliographystyle{amsbracket}
\providecommand{\bysame}{\leavevmode\hbox to3em{\hrulefill}\thinspace}

\end{document}